\documentclass[11pt,reqno]{amsart}
\usepackage{amsmath, amssymb, amsthm}
\usepackage{url}
\usepackage{hyperref}
\usepackage{tikz}
\usepackage{ytableau}
\usepackage[utf8]{inputenc}
\usepackage{xcolor}
\usepackage{tikz-cd}
\usepackage{hyperref}
\usepackage{nccmath}

\usepackage{youngtab}

\usetikzlibrary{decorations.pathreplacing}

\usetikzlibrary{decorations.pathreplacing}

\renewcommand{\(}{\left\(}
\renewcommand{\)}{\right\)}
\renewcommand{\[}{\left\[}
\renewcommand{\]}{\right\]}

\numberwithin{equation}{section}
\theoremstyle{plain}
\newtheorem{theorem}{Theorem}[section]
\newtheorem{lemma}[theorem]{Lemma}

\newtheorem{remark}[theorem]{Remark}

\newtheorem{problem}[theorem]{Problem}

\newtheorem{corollary}[theorem]{Corollary}

\usepackage{mathrsfs}

\makeatletter
\def\proof{\@ifnextchar[{\@oproof}{\@nproof}}
\def\@oproof[#1][#2]{\trivlist\item[\hskip\labelsep\textit{#2 Proof of\
		#1.}~]\ignorespaces}
\def\@nproof{\trivlist\item[\hskip\labelsep\textit{Proof.}~]\ignorespaces}

\makeatother

\makeatletter
\@namedef{subjclassname@2020}{%
	\textup{2020} Mathematics Subject Classification}
\makeatother

\begin{document}
	
	\title[Shifted finite differences of the overpartition function]{Asymptotics of the shifted finite differences of the overpatition function and a problem of Wang--Xie--Zhang} 
	%{\Large }
	
	\author{Gargi Mukherjee}
	\address{School of Mathematical Sciences,
		National Institute of Science Education and Research, Bhbaneswar, An OCC of Homi Bhabha National Institute,  Via- Jatni, Khurda, Odisha- 752050, India}
	\email{gargi@niser.ac.in}
	
	\subjclass[2020]{05A16, 05A20, 11P82.}
	\keywords{Overpartitions, Finite differences, Asymptotic expansion}	
	\maketitle
	
	\begin{abstract}
Let $\overline{p}(n)$ denote the overpartition function, and for $j\in \mathbb{N}$, $\Delta^r_j$ denote the $r$-fold applications of the shifted difference operator $\Delta_j$ defined by $\Delta_j(a)(n):=a(n)-a(n-j)$. The main goal of this paper is to derive an asymptotic expansion of $\Delta^r_j(\overline{p})(n)$ with an effective error bound which subsequently gives an answer to a problem of Wang, Xie, and Zhang. In order to get the asymptotics of $\Delta^r_j(\overline{p})(n)$, we  derive an asymptotic expansion of the shifted overpartition function $\overline{p}(n+k)$ for any integer $k\neq 0$. 

	\end{abstract}

\section{Introduction}\label{intro}
An overpartition of a positive integer $n$, introduced by Corteel and Lovejoy \cite{CorteelLovejoy}, is a non-increasing sequence of positive integers whose sum is $n$ in which the first occurrence of a number may be overlined, $\overline{p}(n)$ denotes the number of overpartitions of $n$, and $\overline{p}(0):=1$. For instance, there are $8$ overpartitions of $3$ enumerated by $$3, \overline{3}, 2+1, \overline{2}+1, 2+\overline{1}, \overline{2}+\overline{1},1+1+1, \overline{1}+1+1.$$
Analogous to the celebrated Hardy-Ramanujan-Rademacher formula for the partition function $p(n)$ (see \cite{HardyRamanujan,Rademacher}), Zuckerman's \cite{Zuckerman} formula for $\overline{p}(n)$ states that
\begin{equation}\label{Zuckerman}
\overline{p}(n)=\frac{1}{2\pi}\underset{2 \nmid k}{\sum_{k=1}^{\infty}}\sqrt{k}\underset{(h,k)=1}{\sum_{h=0}^{k-1}}\dfrac{\omega(h,k)^2}{\omega(2h,k)}e^{-\frac{2\pi i n h}{k}}\dfrac{d}{dn}\left(\dfrac{\sinh \frac{\pi \sqrt{n}}{k}}{\sqrt{n}}\right),
\end{equation}
where
\begin{equation*}
\omega(h,k)=\text{exp}\left(\pi i \sum_{r=1}^{k-1}\dfrac{r}{k}\left(\dfrac{hr}{k}-\left\lfloor\dfrac{hr}{k}\right\rfloor-\dfrac{1}{2}\right)\right)
\end{equation*}
for $(h,k)\in \mathbb{Z}_{\geq 0}\times\mathbb{Z}_{\geq 1}$. Similar to the work done by Lehmer \cite{Lehmer} for $p(n)$, Engel \cite{Engel} estimated an error term for $\overline{p}(n)$ as
\begin{equation}\label{Engel1}
\overline{p}(n)=\frac{1}{2\pi}\underset{2 \nmid k}{\sum_{k=1}^{N}}\sqrt{k}\underset{(h,k)=1}{\sum_{h=0}^{k-1}}\dfrac{\omega(h,k)^2}{\omega(2h,k)}e^{-\frac{2\pi i n h}{k}}\dfrac{d}{dn}\left(\dfrac{\sinh \frac{\pi \sqrt{n}}{k}}{\sqrt{n}}\right)+R_{2}(n,N),
\end{equation}
with
\begin{equation}\label{Engel2}
\left|R_{2}(n,N)\right|< \dfrac{N^{5/2}}{\pi n^{3/2}} \sinh \left(\dfrac{\pi \sqrt{n}}{N}\right).
\end{equation}

Let $\Delta$ be the forward difference operator defined by $\Delta(a)(n):=a(n)-a(n-1)$ for a sequence $(a(n))_{n\ge 0}$ and $\Delta^r$ be the $r$-fold applications of $\Delta$ defined as $\Delta^ra(n):=\Delta\left(\Delta^{r-1} a(n)\right)$ for $r\in \mathbb{N}$. In 1977, Good \cite{Good} proposed a problem on $\Delta^r (p)(n)$ which states that for $r\ge 3$, $\left(\Delta^r (p)(n)\right)_{n\ge 1}$ is a sequence of alternating terms for a certain range $1\le n\le n_0(r)$ and then become positive. Using Hardy-Ramanujan-Rademacher \cite{HardyRamanujan,Rademacher} formula for $p(n)$, Gupta proved that $\Delta^r (p)(n)>0$ for all $n>n_0(r)$ but without any explicit estimate for $n_0(r)$ and conjectured that $n_0(r)\sim r^3$ as $r\rightarrow\infty$. Odlyzko \cite{Odlyzko} settled the Good-Gupta conjecture by proving that $(-1)^n \Delta^r (p)(n)\ge 0$ for all $1\le n<n_0(r)$ and $\Delta^r (p)(n)\ge 0$ for $n\ge n_0(r)$ with $n_0(r)\sim \frac{6}{\pi^2}r^2 (\log r)^2$ as $r\rightarrow\infty$. Remarkably, an exact equation for $n_0(r)$ was given by Knessl and Keller \cite{KnesslKeller} which subsequently refined the asymptotic estimation of $n(r)$ by Odlyzko mentioned before. Later Almkvist \cite{Almkvist} obtained an exact formula for $\Delta^r(p)(n)$. Recently, Gomez, Males, and Rolen \cite{GomezMalesRolen} considered a generalized version of the difference operator $\Delta$ and defined {\it $j$-shifted difference} by $\Delta_j (a)(n):=a(n)-a(n-j)$ (note that, for $j=1$, we retrieve the classical $\Delta$ operator) and obtained an asymptotic expansion (up to first 4 terms) of $\Delta^r (p)(n)$ along with an explicit estimation of the error term (cf. \cite[Theorem 1.1]{GomezMalesRolen}) and consequently proved that $\Delta^2_j (p)(n)>0$ for all $n> 16j^2+1/24$ (cf. \cite[Theorem 1.2]{GomezMalesRolen}). A combinatorial proof of $\Delta^2_jp(n)>0$ is given independently by Gomez, Males, and Rolen \cite{GomezMalesRolen} and Banerjee \cite{Banerjee2}. Banerjee \cite[Section 7.2]{Banerjee0} studied the asymptotic expansion of $\Delta^r_j(p)(n)$ but without any estimation of the error bound. Regarding the logarithmic variation of finite differences of the partition function, denoted by $(-1)^{r-1}\Delta^r(\log p)(n)$, the positivity phenomena similar to the difference of partition function, first studied in detail by Chen, Wang, and Xie \cite[Theorem 3.1 and Theorem  4.1]{ChenWangXie} and settled a conjecture of DeSalvo and Pak \cite[Conjecture 4.3]{DeSalvoPak}. Banerjee \cite[Theorem 4.9]{Banerjee1} obtained an asymptotic expansion for $(-1)^{r-1}\Delta^r(\log p)(n)$.

In context of the overpartition function, Wang, Xie, and Zhang \cite{WangXieZhang} studied the positivity of the higher order differences $\Delta^r(\overline{p})(n)$ (cf. \cite[Theorems 2.1 and 2.3]{WangXieZhang}) and its logarithmic variation $(-1)^{r-1}\Delta^r(\log \overline{p})(n)$ (cf. \cite[Theorem 3.1 and Theorem 4.1]{WangXieZhang}). The author \cite{Mukherjee} settled a problem of Wang, Xie, and Zhang \cite[Problem 3.4]{WangXieZhang} by proving an asymptotic lower bound of $(-1)^{r-1}\Delta^r(\log \overline{p})(n)$. A combinatorial proof of $\Delta^2_j(\overline{p})(n)>0$ is given in \cite{Banerjee2}. The key motivation of this paper arises from the following problem of Wang, Xie, and Zhang \cite[Problem 2.4]{WangXieZhang}. Before we state the problem, let us state the result on $\Delta^r (\overline{p})(n)$ by Wang, Xie, and Zhang.
\begin{theorem}\cite[Theorem 2.1 and Theorem 2.3]{WangXieZhang}
For $r\ge 1$, there exists a positive integer $n(r)$ such that for all $n\ge n(r)$,
\begin{equation}\label{wangxiezhangineq}
0<\Delta^r \left(\overline{p}\right)(n)<2^{r-3}\left(1-2^{-\frac{3}{2}}\right)\zeta\left(\frac{3}{2}\right)\frac{e^{\pi\sqrt{n+r}}}{n+r},
\end{equation}	
where $\zeta$ is the Riemann zeta function.
\end{theorem}
This result led Wang, Xie, and Zhang to propose the following problem.
\begin{problem}\label{wangxiezhangprob}\cite[Problem 2.4]{WangXieZhang}
Find a sharp lower bound of $\Delta^r (\overline{p})(n)$. 
\end{problem}
In other words, it is equivalent to find the asymptotic growth of $\Delta^r (\overline{p})(n)$ as $n\rightarrow\infty$. 

In attempting to give an affirmative answer to this question, we first derive an asymptotic expansion of $\overline{p}(n+k)$ for any $k\in \mathbb{Z}\setminus\{0\}$ in Theorem \ref{thm1.1}. Using Theorem \ref{thm1.1}, we obtain the asymptotic expansion of the $r$-fold applications of the $j$-shifted differences of the overpartition function in Theorem \ref{thm1.2} which in turn gives a sharper lower and upper bounds for $\Delta^r_j(\overline{p})(n-j)$. %(see Corollary \ref{cor1}) and in particular, for $j=1$, in Corollary \ref{cor2} we have refined bounds for $\Delta^r_j\overline{p}(n)$ which answers to the Problem \ref{wangxiezhangprob}. 
%It is worthwhile to mention that, the bounds obtained for $\Delta^r \overline{p}(n)$ in Corollary \ref{cor2} are substantial refinement of bounds obtained in Theorem \ref{wangxiezhangineq}. 

Before we state the main results, for $(N,j,r)\in \mathbb{N}^3$ and $k\in \mathbb{Z}\setminus\{0\}$, let us define the following constants depending on parameters $N,j,r,k$ which will be used in stating the main results.
\begingroup
\allowdisplaybreaks
\begin{align}\label{defn1}
N(m) &:=\begin{cases}
	9, &\quad \text{if}\ m=1,\\
10m \log m-m \log \log m, & \quad \text{if}\ m \geq 2,
\end{cases}\\\label{N_0(m)}
N_0(m)&:=\left\lceil \Bigl(\frac{N(m)}{\pi}\Bigr)^2 \right\rceil,\\\label{defn2}
N_1(m,k)&:=\underset{k\in \mathbb{Z}\setminus\{0\}}{\min}\{N_0(m)-k\},\\\label{n(N,k)}
n(N,k)&:=
\begin{cases}
\max \Bigl  \{N_0(N+1),\ 4k, \ k^2, \ \Bigl \lceil \frac{k^2\pi^2}{4} \Bigr \rceil \Bigr \}, &\quad \text{if}\ k>0,\\
\max \Bigl  \{N_0(N+1),\ 4|k| \Bigr \}, & \quad \text{if}\ k<0,
\end{cases}\\\label{errorevencase1}
\widetilde{E}_1(N,k)&:=
\frac{2|k|^{\frac{N}{2}}\sinh(\pi \sqrt{|k|})}{\pi^{3/2}} \Bigl({1+\sqrt{2(N+1)}}\Bigr),\\\label{erroroddcase1}
\widetilde{E}_2(N,k)&:=
\frac{2|k|^{\frac{N}{2}}\cosh(\pi \sqrt{|k|})\sqrt{2(N+2)}}{\pi^{3/2}},\\\label{error1}
\overline{E}_1(N,k)&:=\widetilde{E}_1(N,k)+\widetilde{E}_2(N,k),\\\label{error2}
\widehat{E}_2(N)&=\widehat{E}_2(N,k):=
\begin{cases}
\frac{1+\pi}{\pi^{N+1}}, &\quad \text{if}\ k>0,\\
\frac{2^{\frac{N+3}{2}}}{\pi^{N+1}}, & \quad \text{if}\ k<0,
\end{cases}\\\label{E(N,k)}
E(N,k)&:=\overline{E}_1(N,k)+\widehat{E}_2(N),\\\label{n(N,r,j)}
n(N,r,j)&:=\underset{0\le m\le r}{\max} \{n(N,-(m+1)j)\}\\\label{E(N,r,j)}%\label{C_r}
%C_r &:=
%\begin{cases}
%\binom{r}{\frac{r}{2}} (r+1), &\quad \text{if $r$ is even},\\
%\Bigl(	\binom{r}{\frac{r+1}{2}}+\binom{r}{\frac{r-1}{2}}\Bigr) (r+1), & \ \ \text{if $r$ is odd},
%\end{cases}\\\label{deldifferror1}
%E_1(N,r,j)&:=\frac{C_r(r.j)^{\frac{N}{2}}\cosh(\pi\sqrt{r.j})2\sqrt{2(N+2)}}{\pi^{3/2}},\\\label{deldifferror2}
%E_2(N,r,j)&:=\frac{C_r(r.j)^{\frac{N}{2}}\sinh(\pi\sqrt{r.j})2(\sqrt{2(N+1)}+1)}{\pi^{3/2}},\\
E(N,r,j)&:=\sum_{m=0}^{r}\binom{r}{m}E\left(N,-(m+1)j\right).
\end{align} 	
\endgroup 
\begin{theorem}\label{thm1.1}
		Let $n(N,k)$, $E(N,k)$ be as in (\ref{n(N,k)}) and (\ref{E(N,k)}) respectively.. For $N\in \mathbb{N}$, $k \in \mathbb{Z}\setminus\{0\}$, and $n \ge n(N,k)$, we have \begin{equation}\label{eqn thm1.1.}
		\overline{p}(n+k)=\frac{e^{\pi\sqrt{n}}}{8n}\Biggl(\sum_{t=0}^{N} \frac{A_{k}(t)}{n^{\frac{t}{2}}}+O_{\le E(N,k)}\left(n^{-\frac{N+1}{2}}\right)\Biggr),
		\end{equation}
		where \begin{equation}\label{A_{k}(t)}
		A_{k}(t)=\Bigl(\frac{k}{2}\Bigr)^t\sum_{\ell=0}^{\lfloor \frac{t+1}{2} \rfloor} \frac{(-1)^\ell (t+1-\ell)}{(t+1-2\ell)!} \binom{t+1}{\ell}\frac{\pi^{t-2\ell}}{k^\ell}.
		\end{equation}
	\end{theorem}
	\begin{theorem}\label{thm1.2}
	Let $n(N,r,j)$ and $E(N,r,j)$ be as defined in (\ref{n(N,r,j)}) and (\ref{E(N,r,j)}) respectively. For all $j,r \in \mathbb{N}$, $N \in \mathbb{Z}_{\ge r}$ and $n \ge n(N,r,j)$ we have
		\begin{equation}\label{eqnthm1.2.}
		\Delta^r_{j}\left(\overline{p}\right)(n-j)=\frac{e^{\pi\sqrt{n}}}{8n}\Biggl(\sum_{t=r}^{N}\frac{A_j(t,r)}{n^{\frac{t}{2}}}+O_{\le E(N,r,j)}(n^{-\frac{N+1}{2}}) \Biggr),
		\end{equation}
		where \begin{equation}\label{A(t,r)}
		A_j(t,r)=\Bigl(-\frac{1}{2}\Bigr)^t \sum_{\ell=0}^{\lfloor \frac{t+1}{2} \rfloor} \binom{t+1}{\ell}\frac{j^{t-\ell}\pi^{t-2\ell}(t+1-\ell)}{(t+1-2\ell)!} \sum_{m=0}^{r}(-1)^m \binom{r}{m}(m+1)^{t-\ell}.
		\end{equation}
\end{theorem}
\begin{remark}
In particular, we see that 
	\begin{equation}\label{neweqn2}
A_j(r,r)=\left(\frac{\pi j}{2}\right)^r.
\end{equation}
\end{remark}
%	and in particular,
%	\begin{equation}\label{neweqn2}
%	A_j(r,r)=\left(\frac{\pi\cdot j}{2}\right)^r.
%	\end{equation}
%	\end{theorem}
From Theorem \ref{thm1.2}, we have the following asymptotic growth of $\Delta^r_{j}(\overline{p})(n)$.
\begin{corollary}\label{cor1}
	We have
		\begin{equation*}
		\Delta^r_{j}\left(\overline{p}\right)(n)\sim \left(\frac{\pi j}{2}\right)^r\frac{e^{\pi\sqrt{n}}}{8\ n^{\frac{r}{2}+1}}\ \ \text{as}\ \ n\to\infty.
		\end{equation*}
	\end{corollary}
Setting $j=1$ in Corollary \ref{cor1}, we get the asymptotic major term of $\Delta^r(\overline{p})(n)$ stated below.
\begin{corollary}\label{cor2}
We have
	\begin{equation*}
	\Delta^r\left(\overline{p}\right)(n)\sim \left(\frac{\pi}{2}\right)^r\frac{e^{\pi\sqrt{n}}}{8\ n^{\frac{r}{2}+1}}\ \ \text{as}\ \ n\to\infty.
	\end{equation*}
\end{corollary}

The paper is organized in the following way. In Section \ref{sec0}, we present some preparatory lemmas in deriving asymptotic expansion of the shifted overpartition function which will be helpful to get error estimations carried out in Section \ref{proof} that finally lead to the proof of Theorem \ref{thm1.1} and applying Theorem \ref{thm1.1}, subsequently, we will present the proof of Theorem \ref{thm1.2}. Finally, we conclude the paper with raising some open questions in Section \ref{conclusion}.

	\section{Preliminary lemmas}\label{sec0} 
	Following the sketch done in \cite[Lemma 2.1.]{MZZ}, we have the folowing lemma.
	\begin{lemma}\label{lemma1}
	Let $N(m)$ be as defined in \eqref{defn1}.	Then for any integer $m\ge1$,  
	\begin{equation*}
		10x^m e^{-\frac{x}{2}}<1, \ \ \text{for} \ x\ge N(m).  
	\end{equation*}
\begin{proof}
	For $m=1$, it is clear that $N(m)=9$. For $m \ge 2$, define $f(x):=\frac{x}{2}-m\log x -\log 10$. Now, $f^{'}(x)=\frac{1}{2}-\frac{m}{x}>0$ implies that $f(x)$ is increasing for $x >2m$. In order to show $f(x)>0$ for all $x \ge N(m)$, it is enough to show $f(x)>f(N(m))>0$ for all $x \ge N(m)>2m$; i.e., we first show that (i) $N(m)>2m$ and then we show (ii) $f(N(m))>0$. It is easy to verify that  $N(m)>2m$, that is, $10m \log m-m \log \log m >2m$ for $m\ge 2$. Next, in order to prove $f(N(m))>0$, it is equivalent to show that $$5m\log m-\frac{m\log \log m}{2}-m\log(10m \log m-m \log \log m)-\log 10>0,$$
	which in turn, amounts to prove
	$$\log m> \frac{2m+1}{8m}\log 10+\frac{3}{8}\log \log m+\frac{1}{8}\log\Bigl(1-\frac{\log \log m}{10 \log m}\Bigr).$$
	Since $\frac{2m+1}{8m}>\frac 14$ and $\log\left(1-\frac{\log \log m}{10 \log m}\right)<0$ for $m\ge 3$, we have $m>10^{\frac{1}{4}}\log m^{\frac{3}{8}}$ which is equivalent to the inequality stated before. 
	%	$\Leftrightarrow \log m> \frac{2m+1}{8m}\log 10+\frac{3}{8}\log 
%		\log m+\frac{1}{8}\log\Bigl(1-\frac{\log \log m}{10 \log m}\Bigr) \Leftrightarrow m>10^{\frac{1}{4}}\log m^{\frac{3}{8}} \ (\text{since} \ \frac{2m+1}{8m}>\frac{1}{4}, \ \forall m\ge 1 \ \text{and} \  \log\Bigl(1-\frac{\log \log m}{10 \log m}\Bigr)<0, \ \forall \ m \ge 3)$  which is true for all $m \ge 1$.	
\end{proof}
\end{lemma}

Set $\overline{\mu}(n):=\pi\sqrt{n}$.

\begin{lemma}\label{lemma2}
Let $N_0(m)$ be as defined in \eqref{N_0(m)}. Then for all $k \in \mathbb{Z} \setminus\{0\}$, $m\ge 2$ and $n \ge N_0(m)$, 
	\begin{equation}\label{eqn_nplusk}
	\overline{p}(n+k)=\frac{e^{\overline{\mu}(n+k)}}{8(n+k)}\Biggl(1-\frac{1}{\overline{\mu}(n+k)}+O_{\le 1}(\overline{\mu}(n+k)^{-m})\Biggr).
	\end{equation}
\end{lemma}
\begin{proof}
		From \cite[equation (3.5)]{LiuZhang}, we have 
	$$\overline{p}(n)=\frac{e^{\overline{\mu}(n)}}{8n} \Biggl(1-\frac{1}{\overline{\mu}(n)} +\tilde{T}(n)\Biggr)$$ with $$\tilde{T}(n)=\Biggl(1+\frac{1}{\overline{\mu}(n)}\Biggr)e^{-2\overline{\mu}(n)}+\frac{8n}{e^{\overline{\mu}(n)}}R_2(n,2),$$ and $R_2(n,2)$ is the error term of (\eqref{Engel1}), given in (\eqref{Engel2}). 
	By \cite[equation (3.6)]{LiuZhang} we have $$ \mid \tilde{T}(n) \mid<10 e^{-\frac{1}{2}\overline{\mu}(n)}.$$ Using Lemma \ref{lemma1} we get for all $m \ge 2$ and $\overline{\mu}(n) \ge N(m)$; i.e, for $n \ge N_0(m)$, 
	\begin{equation}\label{neweqn1}
	\overline{p}(n)=\frac{e^{\overline{\mu}(n)}}{8n}\Biggl(1-\frac{1}{\overline{\mu}(n)}+O_{\le 1}(\overline{\mu}(n)^{-m})\Biggr).
	\end{equation}
	Thus applying \eqref{neweqn1} with $n\mapsto n+k$ and following \eqref{defn2}, we have for all $k \in \mathbb{Z} \setminus\{0\}$, $m\ge 2$ and $n \ge N_1(m,k)$,
	\begin{equation*}
	\overline{p}(n+k)=\frac{e^{\overline{\mu}(n+k)}}{8(n+k)}\Biggl(1-\frac{1}{\overline{\mu}(n+k)}+O_{\le 1}(\overline{\mu}(n+k)^{-m})\Biggr).
	\end{equation*}
	This finishes the proof of the lemma.
\end{proof}

To derive the coefficient sequence arising from the Taylor expansion of the asymptotic main term $\frac{e^{\overline{\mu}(n+k)}}{8(n+k)}\Bigl(1-\frac{1}{\overline{\mu}(n+k)}\Bigr)$ after extracting out the factor $\frac{e^{\pi \sqrt{n}}}{8n}$, we closely follow the methodology carried out in \cite[Proposition 4.4]{Sullivan}.

\begin{lemma}\label{seclem1}
	For $k\in \mathbb{Z}\setminus\{0\}$, we have
	$$\frac{e^{\overline{\mu}(n+k)}}{8(n+k)}\Biggl(1-\frac{1}{\overline{\mu}(n+k)}\Biggr)=\frac{e^{\pi \sqrt{n}}}{8n}\sum_{t\ge 0}\frac{A_k(t)}{n^{\frac t2}},$$
	with
	$$A_k(t):=\left(\frac{\pi k}{2}\right)^t\sum_{\ell=0}^{\lfloor \frac{t+1}{2} \rfloor}\left(-\frac{1}{\pi^2k}\right)^{\ell} \frac{(t-\ell+1)}{(t+1-2\ell)!}\binom{t+1}{\ell}.$$
\end{lemma}

\begin{proof}
	We first estimate $\frac{e^{\overline{\mu}(n+k)}}{8(n+k)}\Bigl(1-\frac{1}{\overline{\mu}(n+k)}\Bigr)$ up to some truncation point say $N \in \mathbb{N}$. \\
	Setting $z:=\frac{1}{\sqrt{n}}$ and $c:=\pi$, we have 
	\begingroup
	\allowdisplaybreaks
	\begin{align}\label{eqn1}
	\frac{e^{\overline{\mu}(n+k)}}{8(n+k)}\Biggl(1-\frac{1}{\overline{\mu}(n+k)}\Biggr)&= \frac{e^{\pi \sqrt{n}}}{8n} \Biggl(\frac{e^{\pi \sqrt{n}\Bigl(\sqrt{1+\frac{k}{n}}-1\Bigr)}}{(1+\frac{k}{n})}\Biggl(1-\frac{1}{\pi \sqrt{n}\sqrt{1+\frac{k}{n}}}\Biggr)\Biggr)\\\nonumber
	&=\frac{e^{\pi \sqrt{n}}}{8n} \Biggl[e^{\pi \sqrt{n}\Bigl(\sqrt{1+\frac{k}{n}}-1\Bigr)} \Biggl(\frac{1}{1+\frac{k}{n}}-\frac{1}{\pi \sqrt{n}(1+\frac{k}{n})^{3/2}}\Biggr)\Biggr]\nonumber\\
	&= \frac{e^{\pi \sqrt{n}}}{8n}\Biggl[\exp\Bigl(\frac{c}{z}(\sqrt{1+kz^2}-1)\Bigr)\Biggl(\frac{1}{1+kz^2}-\frac{1}{\frac{c}{z}(1+kz^2)^{3/2}}\Biggr)\Biggr].\nonumber 
	\end{align}
	\endgroup
	Let $\alpha:=c\sqrt{n}$, $x:=kz^2=\frac{k}{n}$. Considering the left hand side of \eqref{eqn1} as a formal power series in $x$, we get
	\begin{equation}\label{eqn2}
	\exp(\alpha(\sqrt{x+1}-1))\Biggl(\frac{1}{1+x}-\frac{1}{\alpha(1+x)^{3/2}}\Biggr)=\sum_{j=0}^{\infty}\xi_{j}(\alpha)x^j.
	\end{equation} 
	Integrating \eqref{eqn2} twice with respect to $x$ yields
	\begin{equation*}
	\xi_j(\alpha)=(j+1)(j+2)[x^{j+2}]\frac{4}{\alpha^2}\sum_{\ell=0}^{\infty}\frac{\alpha^\ell}{\ell !}(\sqrt{1+x}-1)^{\ell}.
	\end{equation*}
	Let $y:=\sqrt{1+x}$ and $\mathcal{B}_{-1}(x)$ be the generalized binomial series (see \cite[page 203]{GrahamKnuthPatashnik}). Then for all $m\in \mathbb{C}$ with $m \notin \mathbb{Z}_{\ge 0}$, it follows that 
	\begin{equation}\label{eqn3}
	\Bigl(\frac{1+y}{2}\Bigr)^m=\Bigl(\frac{1+\sqrt{1+x}}{2}\Bigr)^m= {\mathcal{B}_{-1}} \Bigl(\frac{x}{4}\Bigr)^m=\sum_{j=0}^{\infty}\frac{m}{m-j}\binom{m-j}{j}\frac{x^j}{4^j}.
	\end{equation}
	Therefore, 
	\begingroup
	\allowdisplaybreaks
	\begin{align*}
	\xi_j(\alpha)&=(j+1)(j+2)\sum_{\ell=0}^{\infty}\Bigl(\frac{\alpha}{2}\Bigr)^{\ell-2}\frac{1}{\ell!}[x^{j+2-\ell}]\Bigl(\frac{1+y}{2}\Bigr)^{-\ell}\\ 
	& =(j+1)(j+2)\sum_{\ell=0}^{j+1} \Bigl(\frac{\alpha}{2}\Bigr)^{j-\ell}\frac{1}{(j+2-\ell)!}[x^{\ell}]\Bigl(\frac{1+y}{2}\Bigr)^{-(j+2-\ell)}\\ 
	&= (j+1) \sum_{\ell=0}^{j+1} \Bigl(\frac{\alpha}{2}\Bigr)^{j-\ell}\frac{1}{(j+2-\ell)!}\frac{j+2-\ell}{j+2}\binom{-j-2}{\ell}\frac{1}{4^{\ell}}\\ 
	%	&= (j+1) \sum_{\ell=0}^{j+1} \Bigl(\frac{\alpha}{2}\Bigr)^{j-\ell}\frac{1}{(j+2-\ell)!}\frac{j+2-\ell}{j+2}\frac{(j+2)\ldots(j+\ell+1)(-1)^{\ell}}{{\ell}!}\frac{1}{4^{\ell}}\\ 
	&= \sum_{\ell=0}^{j+1} \frac{(j+1)\alpha^{j-\ell}(-1)^{\ell}}{2^{j+\ell}(j+\ell+1)!}\binom{j+\ell+1}{\ell},
	\end{align*}
	\endgroup
	where in the second step, we have used \eqref{eqn3} along with the fact that
	$$
	\Bigl(\frac{1+y}{2}\Bigr)^{-(j+2-\ell)}=0 \ \ \text{for}\ \ \ell \ge {j+2}\ \ \text{as}\ \ \binom{m-j-1}{j-1}=0\ \text{for }\ \left\lfloor\frac m2\right\rfloor +1\le j\le m-1.
	$$
	Since, in \eqref{eqn2}, $\xi_{j}(\alpha)x^j$ contains the term $\alpha^{j-\ell}x^{j}={\pi}^{j-\ell}k^{j}n^{-\frac{j+\ell}{2}}$, thus we have
	\begingroup
	\allowdisplaybreaks
	\begin{align*}
	A_{k}(t)&=\underset{\ell \le j+1}{\sum_{j+\ell =t}^{}}\frac{(-1)^{\ell}(j+1){\pi}^{j-\ell}k^j}{2^{j+\ell}(j+1-\ell)!}\binom{j+\ell+1}{\ell} = \sum_{\ell=0}^{\lfloor \frac{t+1}{2} \rfloor} \frac{(-1)^{\ell}(t-\ell+1)\pi^{t-2\ell}k^{t-\ell}}{2^{t}(t+1-2\ell)!}\binom{t+1}{\ell},
	\end{align*}
	\endgroup
	as $j+\ell=t$ and $\ell \le j+1$, we have $t-j \le j+1 \Leftrightarrow j \ge \frac{t-1}{2}$ which implies $\ell \le t-j \le \frac{t+1}{2}$ and this finally concludes the proof. 
\end{proof}
	%\section{Proof of Theorem \ref{thm1.1}}\label{sec1}
		\section{Proof of Theorem \ref{thm1.1} and Theorem \ref{thm1.2}}\label{proof}

	Before we present the proof of Theorem \ref{thm1.1}, we first need to estimate an error bound of $\sum_{t\ge N+1}\frac{A_k(t)}{n^{\frac t2}}$.
\begin{lemma}\label{seclem2}
Let $n(N,k)$ be as in \eqref{n(N,k)} and $\overline{E}_1(N,k)$ be as in \eqref{error1}. Then for $k\neq 0$, and $n\ge n(N,k)$, we have
$$\sum_{t\ge N+1}^{}\frac{A_{k}(t)}{n^{\frac{t}{2}}}=O_{\le \overline{E}_1(N,k)}\left(n^{-\frac{N+1}{2}}\right).$$
\end{lemma}
\begin{proof}
From Lemma \ref{seclem1}, it follows that for $k> 0$,
\begin{align}\label{evencoeff}
A_k(2t)&=
\Bigl(-\frac{k}{4}\Bigr)^t\displaystyle \sum_{\ell=0}^{t} (-1)^{\ell} \binom{2t+1}{t-\ell}(t+\ell+1) \frac{(\pi \sqrt{k})^{2\ell}}{(2\ell+1)!},\\\label{oddcoeff}
A_k(2t+1)&=\frac{k^t(-1)^{t+1}}{2^{2t+1}\pi} \sum_{\ell=0}^{t+1} (-1)^{\ell} \binom{2t+2}{t-\ell+1}(t+\ell+1) \frac{(\pi \sqrt{k})^{2\ell}}{(2\ell)!}.
\end{align}	
For $t \in \mathbb{N}$, define $c(t,\ell):=\binom{2t+1}{t-\ell}(t+\ell+1)$ and we see that
$$\frac{c(t,\ell)}{c(t,0)}=\frac{\prod_{j=1}^{\ell}(t-j+1)}{\prod_{j=1}^{\ell}(t+j)}\le1,$$ which implies that $c(t,\ell)\le c(t,0)$ for all $0\le\ell\le t$. Consequently, using \eqref{evencoeff} and the fact that $\binom{2t}{t} \le \frac{4^t}{\sqrt{\pi t}}$, we have
\begingroup
\allowdisplaybreaks
\begin{align*}
	\left | A_{k}(2t) \right | & \le  \frac{1}{\pi \sqrt{k}} \Bigl(\frac{k}{4}\Bigr)^t \binom{2t+1}{t}(t+1) \left | \sum_{\ell=0}^{t}\frac{(-1)^{\ell}(\pi\sqrt{k})^{2\ell+1}}{(2\ell+1)!} \right| 
	 \le \frac{k^{t-1/2}(2t+1)}{\pi^{3/2}\sqrt{t}}\sinh(\pi\sqrt{k}),
\end{align*}
\endgroup 
and so
\begingroup
\allowdisplaybreaks
\begin{align}\nonumber
\sum_{t\ge \frac{N+1}{2}}^{} \frac{\left |A_{k}(2t)\right |} {n^{t}}&\le \frac{k^{-\frac{1}{2}}\sinh(\pi \sqrt{k})}{\pi^{3/2}}\sum_{t \ge \frac{N+1}{2}}^{}\frac{2t+1}{\sqrt{t}}\frac{k^t}{n^t}\\\nonumber
 &\le \frac{k^{\frac{N}{2}}\sinh(\pi \sqrt{k})}{\pi^{3/2}}n^{-\frac{N+1}{2}} \sum_{t\ge0}^{}\Biggl(2\sqrt{\frac{N+1}{2}}\sqrt{1+\frac{2t}{N+1}}+\frac{1}{\sqrt{t+\frac{N+1}{2}}}\Biggr) \Bigl(\frac{k}{n}\Bigr)^t \\\nonumber
  & \le \frac{k^{\frac{N}{2}}\sinh(\pi \sqrt{k})}{\pi^{3/2}}n^{-\frac{N+1}{2}} \sum_{t\ge0}^{}\Biggl(2\sqrt{\frac{N+1}{2}}\sqrt{1+t}+\frac{1}{\sqrt{t+1}}\Biggr) \Bigl(\frac{k}{n}\Bigr)^t \\\nonumber
   &\le \frac{k^{\frac{N}{2}}\sinh(\pi \sqrt{k})}{\pi^{3/2}}n^{-\frac{N+1}{2}} \Bigl({1+\sqrt{2(N+1)}}\Bigr) \sum_{t\ge0}^{} \Bigl(\frac{2k}{n}\Bigr)^t\\\label{evenestim1}
  % &\hspace{6.5 cm} \left(\text{since} \ \sqrt{t+1} \le 2^t \ \ \text{for } \ \ t>0\right)\\\label{evenestim1} 
   &\le  \Biggl(\frac{2k^{\frac{N}{2}}\sinh(\pi \sqrt{k})}{\pi^{3/2}} \Bigl({1+\sqrt{2(N+1)}}\Bigr) \Biggr) n^{-\frac{N+1}{2}}=\widetilde{E}_1(N,k)n^{-\frac{N+1}{2}},
\end{align}
\endgroup
using the fact $\sqrt{t+1}\le 2^t$ for $t>0$ in the fourth step and $n\ge n(N,k)\ge 4k$ in the second last step and \eqref{errorevencase1} in the final step.

Similarly, we estimate 
$$\left | A_{k}(2t+1) \right | \le \frac{k^t (t+1)}{4^t 2\pi} \binom{2t+2}{t+1} \cosh(\pi \sqrt{k}) \le \frac{2k^t \sqrt{t+1}}{\pi^{3/2}}\cosh(\pi \sqrt{k}),$$
and consequently using \eqref{oddcoeff}, \eqref{erroroddcase1}, and $n\ge n(N,k)\ge 4k$, it follows that
\begingroup
\allowdisplaybreaks
\begin{align}\nonumber
	\sum_{t\ge \frac{N}{2}}^{}\frac{\left |A_{k}(2t+1) \right |}{n^{t+\frac{1}{2}}}&\le n^{-\frac{1}{2}} \frac{2\cosh(\pi \sqrt{k})}{\pi^{3/2}}\sum_{t\ge \frac{N}{2}}^{}\sqrt{t+1}\frac{k^t}{n^t}\\\label{oddestim1}
	&\le \frac{2k^{N/2}\cosh(\pi \sqrt{k})}{\pi^{3/2}}n^{-\frac{N+1}{2}} \sum_{t\ge 0}^{}\frac{k^t}{n^t}\sqrt{t+\frac{N}{2}+1}\le \widetilde{E}_2(N,k) n^{-\frac{N+1}{2}}.
\end{align}
\endgroup
Combining \eqref{evenestim1} and \eqref{oddestim1}, we get
\begingroup
\allowdisplaybreaks
\begin{align}\nonumber
	\left | \sum_{t\ge N+1}^{}\frac{A_{k}(t)}{n^{\frac{t}{2}}} \right |= \left |  \sum_{t\ge \frac{N+1}{2}}^{}\frac{A_{k}(2t)}{n^{t}} + \sum_{t\ge \frac{N}{2}}^{}\frac{A_{k}(2t+1)}{n^{t+\frac{1}{2}}} \right |&\le   \sum_{t\ge \frac{N+1}{2}}^{} \frac{\left |A_{k}(2t)\right |} {n^{t}} + \sum_{t\ge \frac{N}{2}}^{}\frac{\left |A_{k}(2t+1) \right |}{n^{t+\frac{1}{2}}}\\\label{errorestimcase1}
	&=\overline{E}_1(N,k) n^{-\frac{N+1}{2}}.
\end{align}
\endgroup
This concludes the proof of for the case $k>0$. Next, we proceed to the case where $k<0$. 

Let $k=-m$, $m \in \mathbb{Z}_{>0}$; i.e, $m=|k|$. Then we have
$$A_{-m}(t)=\frac{(-m)^t}{2^t}\sum_{\ell=0}^{\left\lfloor \frac{t+1}{2}\right \rfloor} \frac{(t+1-\ell)}{(t+1-2\ell)!}\binom{t+1}{\ell}\frac{\pi^{t-2\ell}}{m^{\ell}}. $$ In a similar way like in the case $k \in \mathbb{Z}_{>0}$, here we have the following bounds
$$\left | A_{-m}(2t+1) \right | \le  \frac{2m^t \sqrt{t+1}}{\pi^{3/2}}\cosh(\pi \sqrt{m}),\ \left | A_{-m}(2t) \right | \le  \frac{m^{t-\frac{1}{2}} (2t+1)}{\pi^{3/2} \sqrt{t}}\sinh(\pi \sqrt{m}),$$
which consequently implies that for $n\ge n(N,k)$
$$\sum_{t\ge \frac{N}{2}}\frac{\left |A_{k}(2t+1) \right |}{n^{t+\frac{1}{2}}} \le \widetilde{E}_2(N,k) n^{-\frac{N+1}{2}}\ \ \text{and}\ \  
\sum_{t\ge \frac{N+1}{2}}\frac{\left |A_{k}(2t) \right |}{n^{t}}\le \widetilde{E}_1(N,k) n^{-\frac{N+1}{2}}.$$
Analogous to the estimation done in \eqref{errorestimcase1}, in this case, we have also
$$	\left | \sum_{t\ge N+1}^{}\frac{A_{k}(t)}{n^{\frac{t}{2}}} \right |\le \overline{E}_1(N,k) n^{-\frac{N+1}{2}},$$
which concludes the proof.
\end{proof}

\begin{lemma}\label{seclem3}
	Let $n(N,k)$ be as in \eqref{n(N,k)} and $\widehat{E}_1(N,k)$ be as in \eqref{error2}. Then for $k\neq 0$, and $n\ge n(N,k)$, we have
	$$\frac{e^{\pi \sqrt{n+k}}}{8(n+k)} \overline{\mu}(n+k)^{-(N+1)}=\frac{e^{\pi \sqrt{n}}}{8n} O_{\le \widehat{E}_2(N)}\left(n^{-\frac{N+1}{2}}\right).$$
\end{lemma}
\begin{proof}
Before we move on to estimate the error term, we note down the following well known Bernoulli's inequality
\begin{equation}\label{Bernoulli}
(1+x)^r\le 1+rx\ \ \text{for}\ \ 0\le r\le 1\ \ \text{and}\ \ x\ge -1.
\end{equation}
For $k>0$ and $n\ge n(N,k)\ge \max\{k^2,\frac{k^2\pi^2}{4}\}$, and using the fact that $e^x<1+2x$ for $0<x<1$, we bound the error term as
\begingroup
\allowdisplaybreaks
\begin{align*}
	\frac{e^{\pi \sqrt{n+k}}}{8(n+k)} \overline{\mu}(n+k)^{-(N+1)}&=\frac{e^{\pi \sqrt{n}}}{8n} \frac{e^{\pi \sqrt{n}(\sqrt{1+\frac{k}{n}}-1)}}{\pi^{N+1}n^{\frac{N+1}{2}}}\frac{1}{(1+\frac{k}{n})^{\frac{N+3}{2}}}\le \frac{e^{\pi \sqrt{n}}}{8n} \frac{e^{\pi \sqrt{n}(\sqrt{1+\frac{k}{n}}-1)}}{\pi^{N+1}n^{\frac{N+1}{2}}}\\ 
	&\hspace{-2 cm} \le  \frac{e^{\pi \sqrt{n}}}{8n} \frac{e^{\frac{k\pi}{2\sqrt{n}}}}{\pi^{N+1}n^{\frac{N+1}{2}}} \le  \frac{e^{\pi \sqrt{n}}}{8n} \Biggl(\frac{1+\pi}{\pi^{N+1}}\Biggr) n^{-\frac{N+1}{2}}=\widehat{E}_2(N)n^{-\frac{N+1}{2}},
\end{align*}
\endgroup
using \eqref{Bernoulli} in the third step and for $k<0$, we estimate in the following way
\begingroup
\allowdisplaybreaks
\begin{align*}
	\frac{e^{\pi \sqrt{n-k}}}{8(n-k)} \overline{\mu}(n-k)^{-(N+1)} & =  \frac{e^{\pi \sqrt{n}}}{8n} \frac{e^{\pi \sqrt{n}(\sqrt{1-\frac{k}{n}}-1)}}{\pi^{N+1}n^{-\frac{N+1}{2}}}\frac{1}{(1-\frac{k}{n})^{\frac{N+3}{2}}}\le \frac{e^{\pi \sqrt{n}}}{8n} \frac{2^{\frac{N+3}{2}}e^{\pi \sqrt{n}(\sqrt{1-\frac{k}{n}}-1)}}{\pi^{N+1}n^{-\frac{N+1}{2}}}\\ 
	& \le  \frac{e^{\pi \sqrt{n}}}{8n} \frac{2^{\frac{N+3}{2}}e^{-\frac{k\pi}{2\sqrt{n}}}}{\pi^{N+1}n^{-\frac{N+1}{2}}} \le  \frac{e^{\pi \sqrt{n}}}{8n}\widehat{E}_2(N) n^{-\frac{N+1}{2}},
\end{align*}
\endgroup
using \eqref{Bernoulli} in the third step. From the above two estimates done separately for both cases for $k$, we conclude the proof.
\end{proof}

\emph{Proof of Theorem \ref{thm1.1}}
Applying Lemmas \ref{seclem1}, \ref{seclem2}, and \ref{seclem3} to \eqref{lemma2} with \eqref{E(N,k)}, we get for $n\ge n(N,k)$,
\begin{align*}
\overline{p}(n+k)=\frac{e^{\pi \sqrt{n}}}{8n}\left(\sum_{t=0}^{N}\frac{A_k(t)}{n^{\frac t2}}+O_{\le E(N,k)}\left(n^{-\frac{N+1}{2}}\right)\right).
\end{align*}
This finishes the proof.
\qed 

Before we proceed to present the proof of Theorem \ref{thm1.2}, we recall the definition of Stirling number. Adopting the notation from \cite{GrahamKnuthPatashnik}, here $n \brace m$ denotes the Stirling number of second kind which counts the number of ways to partition a set of $n$ elements into $m$ nonempty subsets. Here we state three elementary facts about $n \brace m$. From \cite[Table 264 and equation (6.19)]{GrahamKnuthPatashnik}, we have
\begin{equation}\label{fact}
{n \brace m}=0\ \text{for}\ m>n\ \ \text{and}\ \ m!{n \brace m}=\sum_{k=0}^{m}\binom{m}{k}k^n (-1)^{m-k}.
\end{equation}

Using Theorem \ref{thm1.1}, in this section, we are now ready to prove Theorem \ref{thm1.2}. 

\emph{Proof of Theorem \ref{thm1.2}} Following the definition of the operator $\Delta^{r}_j$, it follows that
$$\Delta^r_{j}\left(\overline{p}\right)(n-j)=\sum_{m=0}^{r}(-1)^m\binom{r}{m}\overline{p}(n-m.j-j)=\sum_{m=0}^{r}(-1)^m\binom{r}{m}\overline{p}(n-(m+1)j),$$%= \sum_{m=0}^{r}(-1)^m\binom{r}{m}\left(\sum_{t\ge 0}\frac{A_{-m.j}(t)}{n^{\frac{t}{2}}} \right),$$
and therefore, with $k\mapsto -(m+1)j$, by Theorem \ref{thm1.1}, we have for $n\ge \underset{0\le m\le r}{\max}\{n(N,-(m+1)j)\}$,
\begingroup
\allowdisplaybreaks
\begin{align}\nonumber
\Delta^r_{j}\left(\overline{p}\right)(n-j)&=\underset{m\equiv 0\ \left(\text{mod}\ 2\right)}{\sum_{m=0}^{r}}\binom{r}{m}\overline{p}(n-(m+1)j)-\underset{m\equiv 1\ \left(\text{mod}\ 2\right)}{\sum_{m=0}^{r}}\binom{r}{m}\overline{p}(n-(m+1)j)\\\nonumber
&\le \frac{e^{\pi\sqrt{n}}}{8n}\underset{m\equiv 0\ \left(\text{mod}\ 2\right)}{\sum_{m=0}^{r}}\binom{r}{m}\left(\sum_{t= 0}^{N}\frac{A_{-(m+1).j}(t)}{n^{\frac{t}{2}}}+\frac{E(N,-(m+1)j)}{n^{\frac{N+1}{2}}}\right)\\\nonumber
&\hspace{1 cm}-\frac{e^{\pi\sqrt{n}}}{8n}\underset{m\equiv 1\ \left(\text{mod}\ 2\right)}{\sum_{m=0}^{r}}\binom{r}{m}\left(\sum_{t= 0}^{N}\frac{A_{-(m+1).j}(t)}{n^{\frac{t}{2}}}-\frac{E(N,-(m+1)j)}{n^{\frac{N+1}{2}}}\right)\\\nonumber
&=\frac{e^{\pi\sqrt{n}}}{8n}\left(\sum_{m=0}^{r}(-1)^m\binom{r}{m}\sum_{t= 0}^{N}\frac{A_{-(m+1).j}(t)}{n^{\frac{t}{2}}}+\sum_{m=0}^{r}\binom{r}{m}\frac{E(N,-(m+1)j)}{n^{\frac{N+1}{2}}}\right)\\\label{neweqn3}
&=\frac{e^{\pi\sqrt{n}}}{8n}\left(\sum_{m=0}^{r}(-1)^m\binom{r}{m}\sum_{t= 0}^{N}\frac{A_{-(m+1).j}(t)}{n^{\frac{t}{2}}}+\frac{E(N,r,j)}{n^{\frac{N+1}{2}}}\right).
\end{align}
\endgroup
Similarly, it follows that
\begin{equation}\label{neweqn4}
\Delta^r_{j}\left(\overline{p}\right)(n-j)\ge \frac{e^{\pi\sqrt{n}}}{8n}\left(\sum_{m=0}^{r}(-1)^m\binom{r}{m}\sum_{t= 0}^{N}\frac{A_{-(m+1).j}(t)}{n^{\frac{t}{2}}}-\frac{E(N,r,j)}{n^{\frac{N+1}{2}}}\right).
\end{equation}
Hence, combining \eqref{neweqn3} and \eqref{neweqn4}, for $n\ge n(N,r,j)$, we get
\begin{equation}\label{neweqn5}
\Delta^r_{j}\left(\overline{p}\right)(n-j)= \frac{e^{\pi\sqrt{n}}}{8n}\left(\sum_{m=0}^{r}(-1)^m\binom{r}{m}\sum_{t= 0}^{N}\frac{A_{-(m+1).j}(t)}{n^{\frac{t}{2}}}+O_{\le E(N,r,j)}\left(n^{-\frac{N+1}{2}}\right)\right).
\end{equation}
Next, following \eqref{A_{k}(t)} with $k\mapsto -(m+1)j$, we simplify
\begingroup
\allowdisplaybreaks
\begin{align}\nonumber 
&\sum_{m=0}^{r}(-1)^m\binom{r}{m}\sum_{t= 0}^{N}\frac{A_{-(m+1).j}(t)}{n^{\frac{t}{2}}}\\\nonumber
&=\sum_{m=0}^{r} (-1)^m \binom{r}{m} \sum_{t=0}^{N} \Bigl(\frac{-(m+1)j}{2\sqrt{n}}\Bigr)^t \sum_{\ell=0}^{\lfloor \frac{t+1}{2} \rfloor} \frac{(t+1-\ell)}{(t+1-2\ell)!}\binom{t+1}{\ell}\frac{\pi^{t-2\ell}}{((m+1)j)^{\ell}}\\\nonumber
&=\sum_{t=0}^{N}\left(\Bigl(-\frac{1}{2}\Bigr)^t \sum_{\ell=0}^{\lfloor \frac{t+1}{2} \rfloor} \binom{t+1}{\ell}\frac{j^{t-\ell}\pi^{t-2\ell}(t+1-\ell)}{(t+1-2\ell)!} \sum_{m=0}^{r}(-1)^m \binom{r}{m}(m+1)^{t-\ell}\right)\frac{1}{n^{\frac t2}}\\\label{neweqn6}
&=\sum_{t=0}^{N}\frac{A_j(t,r)}{n^{t/2}}.
\end{align}
\endgroup 
We observe that for $0\le t\le r-1$,
\begingroup
\allowdisplaybreaks
\begin{align}\label{ident}
\sum_{m=0}^{r}\!(-1)^m \binom{r}{m}(m+1)^{t-\ell}&=\!\!\sum_{m=0}^{r}(-1)^m \binom{r}{m}\sum_{s=0}^{t-\ell}\binom{t-\ell}{s}m^s\nonumber\\
&=(-1)^r r!\sum_{s=0}^{t-\ell}\binom{t-\ell}{s} {s\brace r},
\end{align}
\endgroup
using \eqref{fact} in the final step. As $\ell \ge 0$, so for $s\le t-\ell<r$, using \eqref{fact}, we have
$$\sum_{s=0}^{t-\ell}\binom{t-\ell}{s} {s\brace r}=0,$$%\ \left(\text{since}\ {m\brace n}=0\ \text{for}\ m<n\ \text{by \cite[Table 264]{GrahamKnuthPatashnik}}\right),$$
%since ${m\brace n}=0$ for $m<n$ (see \cite[Table 264]{GrahamKnuthPatashnik}). 
Consequently, \eqref{neweqn6} reduces to
$$\sum_{m=0}^{r}(-1)^m\binom{r}{m}\sum_{t= 0}^{N}\frac{A_{-(m+1).j}(t)}{n^{\frac{t}{2}}}=\sum_{t=r}^{N}\frac{A_j(t,r)}{n^{t/2}}.$$
Hence, with \eqref{neweqn5}, we finish the proof of \eqref{eqnthm1.2.}. 

Next, to prove \eqref{neweqn2}, from \eqref{ident}, we see that for $t=r$ and $r\in \mathbb{N}$,
$$\sum_{m=0}^{r}(-1)^m \binom{r}{m}(m+1)^{r-\ell}=\begin{cases}
0,&\quad \text{if}\ \ell\ge 1\ \text{and}\ \ell=0, 0\le s<r,\\
(-1)^r r!,&\quad \text{if}\ (\ell,s)=(0,r).
\end{cases}$$ 
Hence, with the only choice $(\ell,t)=(0,r)$, from \eqref{neweqn6}, it follows that
\begin{equation*}
A_j(r,r)=(-1)^rr! \Bigl(-\frac{1}{2}\Bigr)^r \frac{j^{r}\pi^{r}(r+1)}{(r+1)!}=\left(\frac{\pi j}{2}\right)^r.
\end{equation*}
This concludes the proof of the theorem.
\qed

%\section{Proof of Theorem \ref{thm1.2}}\label{deldiff}

\section{Conclusion}\label{conclusion}
We conclude this paper by raising the following questions with a brief background on each of the problem. 
\begin{enumerate}
\item Following up the work of Odlyzko \cite{Odlyzko} on the positivity of $\Delta^r(p)(n)$ for $n\ge n(r)$ with $n(r)\sim \frac{6}{\pi^2}r^2(\log r)^2$ as $r\rightarrow\infty$, we ask the following:
\begin{problem}\label{prob1}
What is the asymptotic growth of $n(r)$ as $r\rightarrow\infty$ such that for all $n\ge n(r)$, $\Delta^r(\overline{p})(n)\ge 0$? Can one obtain an exact (and/or simplified) formula for such $n(r)$ as in Almkvist,  Knessl and Keller's work \cite{Almkvist,KnesslKeller} for the overpartition function?	
\end{problem}
\item  Let $\text{PL}(n)$ denotes the plane partition function introduced by MacMahon \cite{M}, and the generating function is given by 
$$\sum_{n=0}^{\infty}\textnormal{PL}(n)q^n=\prod_{n=1}^{\infty}\frac{1}{\left(1-q^n\right)^n}.$$
Knessl \cite{Knessl} studied the positivity of $\Delta^r(\text{PL})(n)$, showed that $\Delta^r(\text{PL})(n)>0$ for $n\ge n(r)$ with $$n(r)\sim \frac{2}{3}\left(3\zeta(3)\right)^{-\frac{1}{2}}(r\log r)^{\frac{3}{2}}\ \text{as}\ r\rightarrow\infty.$$
Recently, thanks to the seminal work of Ono, Pujahari and Rolen \cite{OPR} on the asymptotic expansion of $\text{PL}(n)$, a substantial modification of Wright's formula  \cite{Wr}, we have an explicit estimate of the error bound for the asymptotic expansion of $\text{PL}(n)$. We raise the following problem for $\text{PL}(n)$.  
\begin{problem}\label{prob2}
For $r\in \mathbb{N}$, what is the asymptotic growth of $\Delta^r_j(\textnormal{PL})(n)$ as $n\to \infty$?
\end{problem}
\end{enumerate}

	\begin{center}
		\textbf{Acknowledgements}
	\end{center}
The author would like to thank the anonymous referee for his/her comments that improved
the paper substantially. This work has been carried out as a postdoctoral fellow at school of mathematical sciences, National Institute of Science Education and Research,
Bhubaneswar. The author would like to thank the institute for its hospitality and support.

%\begin{center}
%	\textbf{Data Availability}
%\end{center}
%We hereby confirm that Data sharing not applicable to this article as no datasets were generated or analyzed during the current study.

%\begin{center}
%	\textbf{Ethics Declaration}
%\end{center}
%The author confirms that they have no conflict of interest in connection with this paper.

\end{document}